\documentclass[11pt]{amsart}
\usepackage{graphicx}
\usepackage{amssymb, amsmath}
\newtheorem{theorem}{Theorem}[section]
\newtheorem{corollary}[theorem]{Corollary}

\theoremstyle{definition}

\theoremstyle{remark}

\numberwithin{equation}{section}

\newcommand{\Gf}{\mathrm{der}_{\theta, b}(G, \bullet)}

\begin{document}
\title{The Haar measure of a   profinite $n$-ary group}
\author{\sc M. Shahryari}

\address{M. Shahryari\\
 Department of Mathematics,  College of Sciences, Sultan Qaboos University, Muscat, Oman}

\author{\sc M. Rostami}
\thanks{{\scriptsize
\hskip -0.4 true cm MSC(2010):  20N15
}}

\address{M. Rostami\\
Department of Pure Mathematics,  Faculty of Mathematical
Sciences, University of Tabriz, Tabriz, Iran }

\email{m.ghalehlar@squ.edu.om}
\email{m.rostami@tabrizu.ac.ir}
\date{\today}

\begin{abstract}
We prove that every profinite $n$-ary group $(G, f)=\Gf$ has a unique Haar measure $m_p$ and further for every measurable subset $A\subseteq G$, we have
$$
m_p(A)=m(A)=(n-1)m^{\ast}(A)
$$
where $m$ and $m^{\ast}$ are  the normalized Haar measures of the profinite groups $(G, \bullet)$ and the Post cover $G^{\ast}$, respectively.
\end{abstract}

\maketitle


\section{Introduction}

A polyadic group is a  natural generalization of the concept of group to the case where the binary operation of group is replaced with an $n$-ary associative operation, one variable linear equations in which have unique solutions. So, in this article, {\em polyadic group} means an $n$-ary group for a fixed natural number $n\geq 2$. These interesting algebraic objects are introduced by Kasner and D\"ornte (\cite{Dor} and \cite{Kas}) and  studied extensively by Emil Post during the first decades of the last century, \cite{Post}. During decades, many articles have been published regarding the structure of polyadic groups. In our previous work \cite{Shah-Rostam}, we studied profinite polyadic groups: $n$-ary groups which are the inverse limits of finite $n$-ary groups. We proved that a  polyadic topological group $(G, f)$ is profinite if and only if it is compact, Hausdorff, and totally disconnected. Moreover, we showed that for a profinite polyadic group $(G, f)$, its retract $(G, \bullet)$ as well as its Post cover $G^{\ast}$ are  profinite groups. As every profinite group is a compact topological group, there is a unique normalized Haar measure $m$ on $(G, \bullet)$ and another unique normalized Haar measure $m^{\ast}$ on the Post cover. In this note, we show that the polyadic group $(G, f)$ has a unique Haar measure $m_p$ and further, for every measurable subset $A\subseteq G$, we have
$$
m_p(A)=m(A)=(n-1)m^{\ast}(A).
$$
As an application, we will show that for every profinite group $G$, a continues automorphism $\theta$ preserves the Haar measure if $\theta^k$ is an inner automorphism for some $k\geq 1$.

\section{Polyadic groups}
A polyadic group is a pair $(G, f)$ where $G$ is a non-empty set and $f:G^n\to G$ is an $n$-ary operation such that \\

(i) the operation is associative, i.e.
$$
f(x_1^{i-1},f(x_i^{n+i-1}),x_{n+i}^{2n-1})=
f(x_1^{j-1},f(x_j^{n+j-1}),x_{n+j}^{2n-1})
$$
for any $1\leq i<j\leq n$ and for all $x_1,\ldots,x_{2n-1}\in G$, and \\

(ii) for all $a_1,\ldots,a_n, b\in G$ and $1\leq i\leq
n$, there exists a unique element $x\in G$ such that
$$
f(a_1^{i-1},x,a_{i+1}^n)=b.
$$

Note that here we use the compact notation $x_i^j$ for every sequence
$$
x_i, x_{i+1}, \ldots, x_j
$$
of elements in $G$, and in the special case when all terms of this sequence are equal to a fixed $x$, we denote
it by $\stackrel{(t)}{x}$, where $t$ is the number of terms.

Clearly, the case $n=2$ is exactly the definition of ordinary groups. To see various examples of polyadic groups, the reader may look at any of our previous works \cite{Dud-Shah, Khod-Shah, Khod-Shah2, Khod-Shah3, Shah2}. 
Suppose $(G, f)$ is a polyadic group and $a\in G$ is a fixed element. Define a binary operation
$$
x\bullet y=f(x,\stackrel{(n-2)}{a},y).
$$
Then $(G, \bullet)$ is an ordinary group, called the {\em retract} of $(G, f)$ over $a$. Such a retract will be denoted by $\mathrm{ret}_a(G,f)$. All
retracts of a polyadic group are isomorphic. \\

One of the most fundamental theorems on  polyadic groups is the following  {\em Hossz\'{u} -Gloskin's theorem} (see \cite{DG, Hos, Sok} for proof).

\begin{theorem}
Let $(G,f)$ be   a polyadic group. Then there exists an ordinary group $(G, \bullet)$,
an automorphism $\theta$ of $(G, \bullet)$ and  an element $b\in G$ such that\\

1.  $\theta(b)=b$,\\

2.  $\theta^{n-1}(x)=b x b^{-1}$, for every $x\in G$,\\

3.  $f(x_1^n)=x_1\theta(x_2)\theta^2(x_3)\cdots\theta^{n-1}(x_n)b$, for all $x_1,\ldots,x_n\in G$.

\end{theorem}

According to this theorem, we  use the notation $\Gf$ for $(G,f)$
and we say that $(G,f)$ is $(\theta, b)$-derived from the group $(G,\bullet)$. It is known that the group $(G, \bullet)$ is isomorphic to any of the retracts of $(G, f)$.

There is one more important ordinary group associated to a polyadic group which we call it the {\em Post cover}. The proof of this  fundamental theorem as well as the construction of the group $G^{\ast}$ can be found in \cite{Post}.

\begin{theorem}
Let $(G, f)=\Gf$ be a polyadic group. Then there exists a unique group $(G^{\ast}, \circ)$  such that \\

1) The retract $(G, \bullet)$ is isomorphic to a normal subgroup $K$ of $G^{\ast}$.

2) $G$ is contained in $G^{\ast}$ as a coset of $K$.

3) We have $G^{\ast}/K\cong \mathbb{Z}_{n-1}$.

4) Inside $G^{\ast}$, for all $x_1, \ldots, x_n\in G$ we have $f(x_1^n)=x_1\circ x_2\circ \cdots \circ x_n$. 

5) $G^{\ast}$ is generated by $G$.
\end{theorem}

Now, we can go back to profinite $n$-ary groups: a profinite polyadic group is the inverse limit of an inverse system of finite polyadic groups. More precisely, let $(I, \leq)$ be a directed set and suppose $\{ (G_i,f_i), \varphi_{i j}, I\}$ is an inverse system of finite polyadic groups. This means that for every pair $(i, j)$ of elements of $I$ with $j\leq i$, we are given a polyadic homomorphism
$$
\varphi_{i j}:(G_i, f_i)\to (G_j, f_j)
$$
such that the equality $\varphi_{j k}\varphi_{i j}=\varphi_{i k}$ holds for all $k\leq j\leq i$. Now, assume that
$$
(G, f)=\varprojlim_i (G_i, f_i).
$$
Then $(G, f)$ is called a profinite $n$-ary group.  From now on, we consider the pair $(G, f)$ which is the above mentioned inverse limit. 
In \cite{Shah-Rostam} we obtained the following results: \\
 
1) The profinite polyadic group $(G, f)=\Gf$ is compact, Hausdorff, and totally disconnected topological polyadic group. \\

2) The corresponding retract $(G, \bullet)$ and the Post cover $G^{\ast}$ are profinite. \\

3) More generally,  $(G, f)=\Gf$ is profinite if and only if $(G, \bullet)$ is profinite and the automorphism $\theta$ is continues. \\

\section{Haar measure}
We consider a profinite $n$-ary group $(G, f)=\Gf$ with the corresponding Post cover $G^{\ast}$. The definition of a Haar measure on $(G, f)$ is the same as the ordinary topological groups (see \cite{Jarden}) where the invariance is replaced with the following condition:\\

{\em For any measurable subset $A\subseteq G$, any sequence $x_1^n$, and any index $i$, the sets $A$ and $f(x_1^{i-1}, A, x_{i+1}^n)$ have the same measure.}\\

Note that, as we will see soon, we don't need to consider any weaker condition like being $i$-invariant (recall that for compact groups the right and left Haar measures coincide). In what follows we always assume that the measures are normalized. As the ordinary groups $(G, \bullet)$ and $G^{\ast}$ are profinite, they have unique Haar measures $m$ and $m^{\ast}$ respectively. 

\begin{theorem}
The polyadic group $(G, f)$ has a unique Haar measure $m_p$ and further, for every measurable $A\subseteq G$ we have 
$$
m_p(A)=m(A)=(n-1)m^{\ast}(A).
$$
\end{theorem}

\begin{proof}
First, we show the uniqueness. Let $m^{\prime}$ be a Haar measure for $(G, f)$. We show that it is also a Haar measure of $(G, \bullet)$. Recall that $(G, \bullet)\cong \mathrm{ret}_a(G, f)$ for any arbitrary $a\in G$. Suppose $A\subseteq G$ is a measurable set and $x\in G$ is an arbitrary element. Then we have
\begin{eqnarray*}
m^{\prime}(x\bullet A)&=&m^{\prime}(f(x, \stackrel{(n-2)}{a}, A))\\
                      &=&m^{\prime}(A).
\end{eqnarray*}
This shows that $m^{\prime}$ is a Haar measure of $(G, \bullet)$ and by the uniqueness of Haar measure for compact groups, we see that $(G, f)$ must have at most one Haar measure.  

Now, consider a measurable subset $A\subseteq G$ and the Haar measure $m$. For every index $i$, by the Hossz\'{u}-Gloskin's theorem, we  have 
\begin{eqnarray*}
m(f(x_1^{i-1}, A, x_{i+1}^n))&=&m(x_1\bullet \theta(x_2)\bullet \cdots\bullet\theta^{i-2}(x_{i-1})\bullet\\
                             &\ &\ \ \ \ \ \ \theta^{i-1}(A)\bullet\theta^i(x_{i+1})\bullet\cdots\bullet\theta^{n-1}(x_n)\bullet b)\\
                             &=&m(\theta^{i-1}(A)).
\end{eqnarray*}
As a result $m$ is a Haar measure of $(G, f)$ if and only if the automorphism $\theta$ preserves $m$, i.e. $m(\theta(A))=m(A)$ for every measurable $A$. Soon, we will see that this condition is always satisfied. Now, consider the Haar measure $m^{\ast}$. We claim that every measurable subset $A\subseteq G$ is also measurable in $G^{\ast}$. We consider the following cases.\\

1) If $A=G$ then $A$ is measurable in $G^{\ast}$ as $G$ is a coset of retract $(G, \bullet)$ in $G^{\ast}$. Recall that the retract has finite index in $G^{\ast}$. \\

2) If $A$ is open in $G$ then $A=B\cap G$ for some open set $B$ in $G^{\ast}$, so $A$ is measurable in $G^{\ast}$.\\

3) If $A=\bigcup_{i=1}^{\infty}A_i$ and each $A_i$ is measurable in $G$, then by induction each $A_i$ is measurable in $G^{\ast}$ and so is $A$.\\

4) The case $A=\bigcap_{i=1}^{\infty}A_i$ is similar.\\

5) If $A=G\setminus B$ for some measurable $B\subseteq G$, then $B$ is measurable in $G^{\ast}$ by induction and as $A=(G^{\ast}\setminus B)\cap G$, it is measurable in $G^{\ast}$.\\

Note that the index of the retract $K=(G, \bullet)$ in $G^{\ast}$ is $n-1$, and hence, we have 
$$
m^{\ast}(K)=\frac{1}{n-1}=m^{\ast}(G). 
$$
Now, for every measurable set $A$ in $G$, define 
$$
m_p(A)=(n-1)m^{\ast}(A). 
$$
This $m_p$ is a indeed a Haar measure on $(G, f)$ as for any sequence $x_1^n$, any measurable subset $A$, and any index $i$, we have
\begin{eqnarray*}
m_p(f(x_1^{i-1}, A, x_{i+1}^n))&=&(n-1)m^{\ast}(f(x_1^{i-1}, A, x_{i+1}^n))\\
                             &=&(n-1)m^{\ast}(x_1\circ\cdots\circ x_{i-1}\circ A\circ x_{i+1}\circ \cdots\circ x_n)\\
                             &=&(n-1)m^{\ast}(A)\\
                             &=&m_p(A).
\end{eqnarray*}
To check the regularity of $m_p$, suppose again $A\subseteq G$ is a measurable set and $\varepsilon$ is a positive number. Since $A$ is measurable in $G^{\ast}$, there is a closed set $B$ and an open set $U$ in $G^{\ast}$ such that $B\subseteq A\subseteq U$ and $m^{\ast}(U\setminus B)\leq \varepsilon/(n-1)$. Note that $B$ is closed in $G$ as well: $G\setminus B=G\cap(G^{\ast}\setminus B)$ which is open. Let $U_0=G\cap U$. Now we have 
$B\subseteq A\subseteq U_0$ and 
$$
m_p(U_0\setminus B)=(n-1)m^{\ast}(U_0\setminus B)\leq (n-1)m^{\ast}(U\setminus B)\leq \varepsilon
$$
which shows that $m_p$ is regular. 
As a result $(G, f)$ has always a unique Haar measure $m_p$ and consequently we have 
$$
m_p(A)=m(A)=(n-1)m^{\ast}(A).
$$
Beside this, we see that the automorphism $\theta$ preserves the measure $m$. 
\end{proof}

As an application, we may have the following result for profinite groups. 

\begin{corollary}
Let $G$ be a profinite group with Haar measure $m$. Suppose $\theta$ is a continues automorphism of $G$ and there exist some integer $n$ and some element $b$, such that for all $x$ we have $\theta^{n-1}(x)=bxb^{-1}$. Then $\theta$ preserves the measure $m$. 
\end{corollary}

\begin{proof}
Consider the $n$-ary group $(G, f)=\mathrm{der}_{\theta, b}(G, \cdot)$. As $\theta$ is continues, this polyadic group is profinite and hence by the above theorem, $\theta$ preserves the Haar measure $m$.
\end{proof}

{\bf Conflict of interest statement}. There is no conflict of interest.


\begin{thebibliography}{30}



\bibitem{Dor} D\"ornte W., {\it Unterschungen \"uber einen verallgemeinerten Gruppenbegriff}, Math. Z., 1929,  {\bf  29}, pp. 1-19.



\bibitem{Dud2} Dudek W., {\it Remarks on $n$-groups}, Demonstratio Math., 1980, {\bf 13}, pp. 65-181.

\bibitem{DG}  Dudek W.,  Glazek K., {\it Around the Hossz\'u-Gluskin Theorem for $n$-ary groups},  Discrete Math., 2008, {\bf 308}, pp. 4861-4876.



\bibitem{Dud-Shah} Dudek W.,  Shahryari M., {\it Representation theory of polyadic groups}, Algebras and Representation Theory, 2012, {\bf  15}, pp. 29-51.


\bibitem{Hos} Hossz\'{u} M., {\it On the explicit form of $n$-groups}, Publ. Math., 1963, {\bf 10},  pp. 88-92.

\bibitem {Kas} Kasner E., {\it An extension of the group concept}, Bull. Amer. Math. Soc., 1904, {\bf  10}, pp. 290-291.



\bibitem{Khod-Shah} Khodabandeh H., Shahryari M.,
\newblock {\it On the representations and automorphisms of polyadic groups},
\newblock Communications in Algebra, 2012, {\bf 40}, pp. 2199-2212.

\bibitem{Khod-Shah2} Khodabandeh H.,  Shahryari M.,
\newblock {\it Simple polyadic groups},
\newblock Siberian Math. Journal, 2014, {\bf 55}, pp. 734-744.

\bibitem{Khod-Shah3} Khodabandeh H.,  Shahryari M.,
\newblock {\it Equations over  polyadic groups},
\newblock Communication in Algebra, 2017, {\bf 45}, pp. 1227-1238.


\bibitem{Post} Post E., {\it Polyadic groups}, Trans. Amer. Math. Soc., 1940, {\bf  48}, pp. 208-350.

\bibitem{Shah2} Shahryari M., {\it Representations of finite polyadic groups}, Communications in Algebra,  2012, {\bf 40},  pp. 1625-1631.

\bibitem{Shah-Rostam} Shahryari M., Rostami M.,  {\it On profinite polyadic groups}, Siberian Electronic Mathematical Reports,  2023, {\bf 20} (to appear). 

\bibitem{Sok} Sokolov E., {\it On the Gluskin-Hossz\'{u} theorem for Dornte $n$-groups}, Mat. Issled., 1976,  {\bf 39}, pp. 187-189.

\bibitem{Jarden} Fried, M.D., Jarden, M., {\it The Haar Measure (in: Field Arithmetic)},  A Series of Modern Surveys in Mathematics, {\bf 11}, Springer, 2023.


\end{thebibliography}
\end{document}